\documentclass[12pt]{amsart}

\newtheorem{theorem}{Theorem}[section]
\newtheorem{lemma}[theorem]{Lemma}
\newtheorem{corollary}[theorem]{Corollary}

\newtheorem{conjecture}[theorem]{Conjecture}

\theoremstyle{plain}
\newtheorem{definition}[theorem]{Definition}

\newtheorem{remark}[theorem]{Remark}

\newtheorem{idea}[theorem]{Idea}
\newtheorem{setting}[theorem]{Setting}

\theoremstyle{definition}

\theoremstyle{remark}

\numberwithin{equation}{section}

\usepackage{fancyhdr}
\usepackage{amscd}
\usepackage{amsmath}
\usepackage{amsthm}
\usepackage{amssymb}

\begin{document}

\title[K-theory and obstructions]{Chern character and obstructions to deforming cycles}

\author{Sen Yang }

\address{Shing-Tung Yau Center of Southeast University \\ 
Southeast University \\
Nanjing, China\\
}
\address{School of Mathematics \\  Southeast University \\
Nanjing, China\\
}
\email{101012424@seu.edu.cn}

\date{}

\maketitle

\begin{abstract}
Following Bloch-Esnault-Kerz and Green-Griffiths' recent works on deformation of algebraic cycle classes, we use Chern character from K-theory to negative cyclic homology to show how to eliminate obstructions to deforming cycles.
\end{abstract}


\section{Introduction}
The main purpose of this paper is to study obstruction issues in deforming cycles and show how to eliminate obstructions to deforming cycles. To motivate the discussion, we recall the infinitesimal Hodge conjecture.

Let $\mathcal{X}/S$ be a smooth projective scheme, where $S = \mathrm{Spec}(k[[t]])$ with $k$ a field of characteristic zero. For each integer $n \ge 0$, we write $S_n = \mathrm{Spec}\, k[t]/t^{n+1}$ and write $X_n=\mathcal{X} \times _{S} S_n$. Let $K_0 (X_{0})$ and $ H^* _{dR} (X_{0}/k)$ denote the Grothendieck group and de Rham cohomology respectively. There exists a Chern character ring homomorphism
\[
\mathrm{ch}:K_0 (X_{0}) \to H^* _{dR} (X_{0}/k).
\]

For an element $\xi _0 \in  K_0 (X_0)_{\mathbb{Q}}$, we are interested in lifting $\mathrm{ch}(\xi _0) \in H^* _{dR} (X_0/k)$ to $ \mathrm{ch}(\xi ) \in H^* _{dR} (\mathcal{X}/S)$ in the sense that 
\begin{equation*}
\mathrm{ch}(\xi\mid _{X_0})  = \mathrm{ch}(\xi _0)  \in H^* _{dR} (X_0/k),
\end{equation*}
where $\xi \in  K_0 (\mathcal{X}) _{\mathbb{Q}}$.

Let $\nabla: H^{*} _{dR} (\mathcal{X}/S) \to H^{*} _{dR} (\mathcal{X}/S)$ denote the derivation in the parameter $t$ given by the Gauss-Manin connection. There exists a canonical isomorphism
\[
\Phi :  H^{*} _{dR} (\mathcal{X}/S) ^{\nabla} \overset{\sim}{\longrightarrow} H^{*} _{dR} (X_0/k),
\]
 where $H^{*} _{dR} (\mathcal{X}/S) ^{\nabla}$ is the kernel of $\nabla$.

The infinitesimal Hodge conjecture predicts that
\begin{conjecture} [see Conjecture 1.4 of \cite{BEK2}]
\label{C:infHodge}
The following statements are equivalent for an element 
$\xi _0 \in K_0 (X_0)_{\mathbb{Q}}$:
\begin{itemize}
\item [1.] $\Phi ^{-1} \circ \mathrm{ch}(\xi _0) \in  \bigoplus _i H^{2i} _{dR} (\mathcal{X}/S) ^{\nabla}\cap F^i H^{2i} _{dR} (\mathcal{X}/S)$, where $F^i H^{2i} _{dR}$ denotes the Hodge filtration of de Rham cohomology;
\item[2.] there is an element 
$ \xi \in  K_0 (\mathcal{X}) _{\mathbb{Q}}$ such that 
\[
\mathrm{ch}(\xi \mid _{X_0})  = \mathrm{ch}(\xi _0)  \in H^* _{dR} (X_0/k).
\]
\end{itemize}
\end{conjecture}

Some recent progress on the infinitesimal Hodge conjecture has been made by Bloch-Esnault-Kerz \cite{BEK, BEK2}, Green-Griffiths \cite{GGdeformation} and Morrow \cite{M}. Especially relevant to our study of algebraic cycles is the work by Bloch, Esnault and Kerz \cite{BEK2}. They proved that, in appendix A of \cite{BEK2}, the infinitesimal Hodge conjecture is equivalent to the variational Hodge conjecture proposed by Grothendieck \cite{G}. Moreover, motivated by the infinitesimal (variational) Hodge conjecture, they proved the following:
\begin{theorem} [see Theorem 1.2 of \cite{BEK2}]
\label{C:BEK}
Assuming that the Chow-K\"unneth property (part of the standard conjecture) holds, the following statements are equivalent for an element 
$\xi _0 \in K_0 (X_0)_{\mathbb{Q}}$:
\begin{itemize}
\item [1.] $\Phi ^{-1} \circ \mathrm{ch}(\xi _0) \in  \bigoplus _i H^{2i} _{dR} (\mathcal{X}/S) ^{\nabla}\cap F^i H^{2i} _{dR} (\mathcal{X}/S)$;
\item[2.] there is an element 
$\hat \xi \in \left (\varprojlim _n K_0 (X_n) \right )\otimes \mathbb{Q}$ such that 
\[
\mathrm{ch}(\hat{\xi} \mid _{X_0})  = \mathrm{ch}(\xi _0)  \in H^* _{dR} (X_0/k).
\]
\end{itemize}
\end{theorem}
 
 The key point in the proof of this theorem  is to eliminate obstructions to lifting $\mathrm{ch}(\xi _0)$ by using correspondences (the assumption of Chow-K\"unneth property guarantees enough correspondences). Moreover, if the ground field $k$ is algebraic over $\mathbb{Q}$, without assuming the Chow-K\"unneth property,  Bloch, Esnault and Kerz deduced that the obstructions to lifting $\mathrm{ch}(\xi _0)$ can be eliminated.

In the pioneering work \cite{GGtangentspace}, Green and Griffiths studied the deformation of algebraic cycles. In particular, they studied the first order deformations of divisors and zero cycles and then defined their tangent spaces. Dribus, Hoffman and the author extended much of their theory in \cite{DHY, Y-4, Y-2, Y-3}.

Let $X$ be a smooth projective variety over a field $k$ of characteristic $0$ and let $Y \subset X$ be a subvariety of codimension $1$, it is well known that the embedded deformation of the subvariety $Y$ may be obstructed. However, by considering $Y$ as an element of the cycle group $Z^{1}(X)$, Green-Griffiths predicted that we could lift the divisor $Y$ to higher order successively. This was verified by Ng in his Ph.D thesis \cite{Ng} by using the semi-regularity map defined by Kodaira-Spencer \cite{KS} and Bloch \cite{Bloch}.

We sketch Ng's idea briefly. When an infinitesimal deformation of $Y$ is obstructed to higher order,  let $Z$ be a very ample divisor such that $H^{1}(O_{X}(Y+Z))=0$. According to Proposition 1.1 of \cite{Bloch}, the semi-regularity map
\[
\pi: H^{1}(Y \cup Z, N_{Y \cup Z /X}) \to H^{2}(O_{X}),
\]
where $N_{Y \cup Z/X}$ is the normal bundle, agrees with the boundary map in the long exact sequence
\[
 \cdots \to  H^{1}(O_{X}(Y+Z))  \to H^{1}(Y \cup Z, N_{Y \cup Z/X}) \xrightarrow{\pi} H^{2}(O_{X})  \to \cdots.
\]
Since $H^{1}(O_{X}(Y+Z))=0$, the kernel of $\pi$ is $0$, $Y \cup Z$ is semi-regular in $X$. According to Kodaira-Spencer \cite{KS} (see also Theorem 1.2 of \cite{Bloch}), $Y \cup Z$ can be lifted to higher order successively. On the other hand, $Z$ can be always lifted to trivial deformations $Z \times_{\mathrm{Spec}(k)}\mathrm{Spec}(k[\varepsilon]/(\varepsilon^{j+1}))$ successively.

As a cycle, $Y$ can be written as a formal sum
\[
Y= (Y+Z)-Z.
\]
To deform the cycle $Y$ is equivalent to deforming $Y \cup Z$ and $Z$ respectively. Hence, $Y$ lifts to higher order successively, since both $Y \cup Z$ and $Z$ do.

The above method suggests an interesting idea to eliminate obstructions:
\begin{idea} \label{idea: i}
 When the deformation of $Y$ is obstructed, find $Z$ such that
\[
\begin{cases}
1.  \ Z \  \mathrm{helps \ Y \ to \ eliminate \ obstructions},  \\
2. \  Z  \ \mathrm{does \ not \ bring \ new \ obstructions}.
 \end{cases}
\]
\end{idea}

While the deformation of divisors are relatively well understood, it is natural to ask how to go beyond the divisor case. A very interesting work on obstructions to deforming curves on a three-fold had been done by Mukai-Nasu \cite{MN}. Inspired by a question asked by Ng in section 1.5 of \cite{Ng}, the author \cite{Y-5} used K-theory to study the deformation of 1-cycles on a three-fold. For $Y \subset X$ a subvariety of codimension $p$, where $p$ is an integer such that $1 \leq p \leq \mathrm{dim}(X)$, Green-Griffiths \cite{GGtangentspace} (page 187-190) conjectured that we could lift the cycle $Y \in Z^{p}(X)$ to higher order successively. Their conjecture has been answered in Theorem 3.11 of \cite{Y-3}.

This paper is to generalize Idea 1.3 to the study of deformations of cycles codimension $p$. In the second section, we recall background on K-theory and Milnor K-theoretic cycles. In section 3, we show how to eliminate obstructions to deforming cycles of codimension $p$ in Theorem \ref{theorem: mainTheorem}.

We remark that the main result of this paper,  Theorem \ref{theorem: mainTheorem}, is different from Theorem 3.11 of \cite{Y-3}. This is mainly because we do not know whether the map $\mu_{Y}$ of Definition \ref{definition: map1} is surjective or not\footnote{The author thanks Spencer Bloch \cite{Bloch1} for discussions on this issue.}.

\textbf{Notation:}

(1). K-theory in this paper is Thomason-Trobaugh non-connective K-theory, if not stated otherwise. 

(2). For any abelian group $M$, $M_{\mathbb{Q}}$ denotes $M \otimes_{\mathbb{Z}} \mathbb{Q}$. 

(3). If not stated otherwise, $X$ is a smooth projective variety over a field $k$ of characteristic $0$. For each integer $j \geq 0$, $X_{j}$ denotes the $j$-th infinitesimally trivial deformation of $X$, i.e., $X_{j}= X \times_{\mathrm{Spec(k)}} \mathrm{Spec}(k[\varepsilon]/ \varepsilon^{j+1})$.

\section{K-theory and deformation of cycles}
\label{Ng's question and reformulation}
The following setting is used below.
\begin{setting}  \label{setting:s}
Let $Y \subset X$ be an irreducible closed subvariety of codimension $p$, with generic point $y$. Let $W \subset Y$ be an irreducible closed subvariety of codimension 1 in $Y$, with generic point $w$. 

We assume that $W$ is generically defined by $f_{1}, f_{2}, \cdots, f_{p}, f_{p+1}$ and $Y$ is generically defined by $f_{1}, f_{2}, \cdots, f_{p}$. It follows that $O_{X,y}=(O_{X,w})_{Q_{1}}$, where $Q_{1}$ is the ideal $(f_{1}, f_{2}, \cdots, f_{p}) \subset O_{X,w}$. 
\end{setting}

A first order infinitesimal embedded deformation $Y^1 \subset X_{1}$ is generically given by $(f_{1}+ \varepsilon g_{1}, f_{2}+ \varepsilon g_{2}, \cdots, f_{p}+ \varepsilon g_{p})$, where $g_{1}, \cdots, g_{p} \in O_{X,y}$, see \cite{Y-4} (page 711-712) for related discussions if necessary.

Let $F_{\bullet}(f_{1}+\varepsilon g_{1}, \cdots, f_{p}+\varepsilon g_{p})$ denote the Koszul complex associated to the  regular sequence $\{f_{1}+\varepsilon g_{1}, \cdots, f_{p}+\varepsilon g_{p}\}$, which defines an element of the Grothendieck group $K_{0}(O_{X_{1},y} \ \mathrm{on} \ y)_{\mathbb{Q}}$. We recall a map from the Zariski tangent space $\mathrm{T}_{Y}\mathrm{Hilb}(X)$ to the Hilbert scheme at the point $Y$ to the Grothendieck group $K_{0}(O_{X_{1},y} \ \mathrm{on} \ y)_{\mathbb{Q}}$.
\begin{definition}  [Definition 2.4 of \cite{Y-4}] \label{definition: map1}
With notation as above, we define a  map 
\begin{align*}
\mu_{Y}: \mathrm{T}_{Y}\mathrm{Hilb}& (X) \to K_{0}(O_{X_{1},y} \ \mathrm{on} \ y)_{\mathbb{Q}} \\
& Y^1 \longrightarrow  F_{\bullet}(f_{1}+\varepsilon g_{1}, \cdots, f_{p}+\varepsilon g_{p}).
\end{align*}

For $Y \in \mathrm{T}_{Y}\mathrm{Hilb}(X)$, $\mu_{Y}(Y)=F_{\bullet}(f_{1}, \cdots, f_{p})$, where $F_{\bullet}(f_{1}, \cdots, f_{p})$ is the Koszul complex associated to $f_{1}, \cdots, f_{p}$.
\end{definition}

In notation of Setting \ref{setting:s}, let $z$ be the point defined by the prime ideal $Q_{2}=(f_{p+1}, f_{2}, \cdots,f_{p}) \subset O_{X,w}$, then $z \in X^{(p)}$. 
\begin{definition} \label{definition: definingCurveZ}
We define a subscheme $Z \subset X$ to be the Zariski closure of $z$ with closed reduced structure:
\[
Z := \overline{\{z \}}.
\]
\end{definition}
 
 \begin{remark} \label{remark:muz}
 We can similarly define a map 
 \[
 \mu_{Z}: \mathrm{T}_{Z}\mathrm{Hilb}(X) \to K_{0}(O_{X_{1},z} \ \mathrm{on} \ z)_{\mathbb{Q}}
 \] as in Definition \ref{definition: map1}. Let $F_{\bullet}(f_{p+1}, f_{2}, \cdots,f_{p})$ be the
Koszul complex of the sequence $\{f_{p+1}, f_{2}, \cdots,f_{p}\}$. For $Z \in \mathrm{T}_{Z}\mathrm{Hilb}(X) $, $\mu_{Z}(Z)=F_{\bullet}(f_{p+1}, f_{2}, \cdots,f_{p})$. 
 \end{remark}

Recall that Milnor K-groups with support are rationally defined as certain eigenspaces of K-groups in \cite{Y-2}.
\begin{definition}  [Definition 3.2 of \cite{Y-2}] \label{definition:Milnor K-theory with support}
Let $X$ be a finite equi-dimensional noetherian scheme and $y \in X^{(p)}$. For each $m \in \mathbb{Z}$, Milnor K-group with support $K_{m}^{M}(O_{X,y} \ \mathrm{on} \ y)$ is rationally defined to be 
\[
  K_{m}^{M}(O_{X,y} \ \mathrm{on} \ y) := K_{m}^{(m+p)}(O_{X,y} \ \mathrm{on} \ y)_{\mathbb{Q}},
\] 
where $K_{m}^{(m+p)}$ is the eigenspace of $\psi^{l}=k^{m+p}$ and $\psi^{l}$ is the Adams operations.
\end{definition}

Adams operations $\psi^{l}$ for K-theory of perfect complexes defined by Gillet-Soul\'e \cite{GilletSoule} has the following property:
\begin{lemma} [Prop 4.12  of \cite{GilletSoule}]  \label{Lemma: GilletSoule}
Let $L(x_{1}, \cdots, x_{p})$ be the Koszul complex of a regular sequence $x_{1}, \cdots, x_{p}$, then Adams operations $\psi^{l}$ on  $L(x_{1}, \cdots, x_{p})$ satisfy that
\[
\psi^{l}(L(x_{1}, \cdots, x_{p}))  = l^{p}L(x_{1}, \cdots, x_{p}).
\]
\end{lemma}
It follows that $F_{\bullet}(f_{1}+\varepsilon g_{1}, \cdots, f_{p}+\varepsilon g_{p}) \in K_{0}(O_{X_{1},y} \ \mathrm{on} \ y)_{\mathbb{Q}}$ lies in the eigenspace space $K^{(p)}_{0}(O_{X_{1},y} \ \mathrm{on} \ y)_{\mathbb{Q}}$. In other words, $F_{\bullet}(f_{1}+\varepsilon g_{1}, \cdots, f_{p}+\varepsilon g_{p})$ lies in the Milnor K-groups with support:
\[
F_{\bullet}(f_{1}+\varepsilon g_{1}, \cdots, f_{p}+\varepsilon g_{p}) \in K^{M}_{0}(O_{X_{j},y} \ \mathrm{on} \ y).
\]

\begin{theorem} [Theorem 3.14 of \cite{Y-2}]  \label{theorem: firstorder}
For each integer $j>0$, there exists the following commutative diagram in which the morphisms $\mathrm{Ch}$ from K-groups to local cohomology groups are induced by Chern character from K-theory to negative cyclic homology
\begin{equation}
  \begin{CD}
     \bigoplus\limits_{y \in X^{(p)}} H_{y}^{p}((\Omega_{X/ \mathbb{Q}}^{p-1})^{\oplus j}) @<\mathrm{Ch}<< \bigoplus\limits_{y \in X^{(p)}}K^{M}_{0}(O_{X_{j},y} \ \mathrm{on} \ y) \\
      @V\partial_{1}^{p,-p}VV @Vd_{1,X_{j}}^{p,-p}VV  \\
     \bigoplus\limits_{w \in X^{(p+1)}} H_{w}^{p+1}((\Omega_{X/ \mathbb{Q}}^{p-1})^{\oplus j}) @<\mathrm{Ch}< \cong < \bigoplus\limits_{w \in X^{(p+1)}}K^{M}_{-1}(O_{X_{j},w} \ \mathrm{on} \ w)=0. \\
  \end{CD}
\end{equation}
\end{theorem}

Tensor triangular Chow groups of a tensor triangulated category were defined by Balmer \cite{B-5}, and they were further explored by Klein \cite{K}. By slight modifying Balmer's definition, we proposed Milnor K-theoretic cycles.
\begin{definition}[Definition 3.4 and 3.15 of \cite{Y-2}] \label{definition: Milnor K-theoretic Chow groups}
The $p$-th Milnor K-theoretic cycle group of $X$ is defined to be\footnote{It was proved in Theorem 3.16 of \cite{Y-2} that $Z^{M}_{p}(D^{\mathrm{Perf}}(X))$ agreed with the classical cycle group $Z^{p}(X)_{\mathbb{Q}}$.}
\[
Z^{M}_{p}(D^{\mathrm{Perf}}(X)) := \bigoplus\limits_{y \in X^{(p)}} K^{M}_{0}(O_{X,y}  \ \mathrm{on} \ y).
\]

For each integer $j>0$, the $p$-th Milnor K-theoretic cycle group of $X_{j}$ is defined to be\footnote{The reason why we use the kernel of $d_{1,X_{j}}^{p,-p}$ to define Milnor K-theoretic cycles $Z^{M}_{p}(D^{\mathrm{Perf}}(X_{j}))$ is explained in section 2.2 of \cite{Y-3}.}
\[
Z^{M}_{p}(D^{\mathrm{Perf}}(X_{j})) := \mathrm{Ker}(d_{1,X_{j}}^{p,-p}), 
\]
where $d_{1,X_{j}}^{p,-p}$ is the differential in the commutative diagram (2.1).

The elements of $Z^{M}_{p}(D^{\mathrm{Perf}}(X))$ and $Z^{M}_{p}(D^{\mathrm{Perf}}(X_{j}))$ are called Milnor K-theoretic cycles.
\end{definition}

By Lemma \ref{Lemma: GilletSoule}, both $\mu_{Y}(Y)$ and $\mu_{Z}(Z)$ have Adams eigenweight $p$. This shows that
\begin{corollary} \label{corollary:muYZ}
Both $\mu_{Y}(Y)$ and $\mu_{Z}(Z)$ are Milnor K-theoretic cycles
\[
\mu_{Y}(Y) \in Z^{M}_{p}(D^{\mathrm{Perf}}(X)), \ \mu_{Z}(Z) \in Z^{M}_{p}(D^{\mathrm{Perf}}(X)).
\]
\end{corollary}

\begin{remark} \label{Remark:}
It is obvious that $K^{M}_{0}(O_{X,y}  \ \mathrm{on} \ y)$ is a direct summand of $K^{M}_{0}(O_{X_{j},y} \ \mathrm{on} \ y)$ and its image under $d_{1,X_{j}}^{p,-p}$ is zero, so $Z^{M}_{p}(D^{\mathrm{Perf}}(X))$ is a direct summand of $Z^{M}_{p}(D^{\mathrm{Perf}}(X_{j}))$.

\end{remark}

Milnor K-theoretic cycles can detect nilpotents, which is important in the study of deformation of cycles. For each integer $j>0$, the natural map $g_{j}: X_{j-1} \to X_{j}$ 
induces a commutative diagram (see section 3.1 of  \cite{Y-3}),
\[
\begin{CD}
\bigoplus\limits_{y \in X^{(p)}}K^{M}_{0}(O_{X_{j}, y} \ \mathrm{on} \ y)   @>g_{j}^{*}>>
\bigoplus\limits_{y \in X^{(p)}}K^{M}_{0}(O_{X_{j-1},y} \ \mathrm{on} \ y)  \\ 
@Vd_{1,X_{j}}^{p,-p}VV  @Vd_{1,X_{j-1}}^{2,-2}VV \\ 
 \bigoplus\limits_{w \in X^{(p+1)}} K^{M}_{-1}(O_{X_{j},w} \ \mathrm{on} \ w)  @>g_{j}^{*}>>
\bigoplus\limits_{w \in X^{(p+1)}}K^{M}_{-1}(O_{X_{j-1},w} \ \mathrm{on} \ w).
\end{CD}
\]
This further induces 
\begin{equation}
g^{\ast}_{j}:  Z^{M}_{p}(D^{\mathrm{perf}}(X_{j})) \to  Z^{M}_{p}(D^{\mathrm{perf}}(X_{j-1})).
\end{equation}

\begin{definition}  [Definition 3.3 of \cite{Y-3}] \label{definition: deformation}
Given $\xi_{j-1} \in Z^{M}_{p}(D^{\mathrm{perf}}(X_{j-1}))$,  an element $\xi_{j} \in Z^{M}_{p}(D^{\mathrm{perf}}(X_{j}))$ is called a deformation (or lift) of $\xi_{j-1}$,  if $g^{\ast}_{j}(\xi_{j}) = \xi_{j-1}$.
\end{definition}

The elements $\xi_{j-1}$ and $\xi_{j}$ can be formally written as finite sums
\[
 \xi_{j-1}= \sum_{y \in X^{(p)}}\lambda_{j-1}\cdot \overline{\{y \}},  \  \xi_{j} = \sum_{y \in X^{(p)}}\lambda_{j}\cdot \overline{\{ y \}},
\] 
where $\overline{\{y \}}$ is with closed reduced structure and $\lambda_{j}$'s are perfect complexes such that
$\sum\limits_{y}\lambda_{j} \in \mathrm{Ker}(d_{1,X_{j}}^{p,-p}) \subset \bigoplus\limits_{y \in X^{(p)}}K_{0}(O_{X_{j}, y} \ \mathrm{on} \ y)_{\mathbb{Q}}$. When we lift from $\xi_{j-1}$ to $\xi_{j}$, we lift the coefficients from $\sum\limits_{y}\lambda_{j-1}$ to $\sum\limits_{y}\lambda_{j}$.

\section{Chern character and obstructions} 
\label{Curves on 3-fold-Ng's question}

Let $D^{\mathrm{perf}}(X_{j})$ denote the derived category obtained from the exact category of perfect complexes on $X_{j}$ and $\mathcal{L}_{(i)}(X_{j})$ is defined to be
\[
  \mathcal{L}_{(i)}(X_{j}) := \{ E \in D^{\mathrm{perf}}(X_{j}) \mid \mathrm{codim(supph(E))} \geq -i \},
\]
where the closed subset $\mathrm{supph}(E) \subset X$ is the support of the total homology of the perfect complex $E$ and the codimension of $\mathrm{supph}(E)$ is no less than $-i$.

Let $(\mathcal{L}_{(i)}(X_{j})/\mathcal{L}_{(i-1)}(X_{j}))^{\#}$ denote the idempotent completion of Verdier quotient $\mathcal{L}_{(i)}(X_{j})/\mathcal{L}_{(i-1)}(X_{j})$ in the sense of Balmer-Schlichting \cite{B-S}.
 \begin{theorem} [\cite{B-3}] \label{theorem: Balmer theorem}
 For each $i \in \mathbb{Z}$, localization induces an equivalence
\[
 (\mathcal{L}_{(i)}(X_{j})/\mathcal{L}_{(i-1)}(X_{j}))^{\#}  \simeq \bigsqcup_{x \in X^{(-i)}}D_{{x}}^{\mathrm{perf}}(X_{j})
\]
between the idempotent completion of the quotient $\mathcal{L}_{(i)}(X_{j})/\mathcal{L}_{(i-1)}(X_{j})$ and the coproduct over $x \in X^{(-i)}$ of the derived category of perfect complexes of $ O_{X_{j},x}$-modules with homology supported on the closed point $x \in \mathrm{Spec}(O_{X,x})$. Consequently, localization induces an isomorphism
\begin{equation}
K_{0}((\mathcal{L}_{(i)}(X_{j})/\mathcal{L}_{(i-1)}(X_{j}))^{\#})  \simeq \bigoplus_{x \in X^{(-i)}}K_{0}(O_{X_{j},x} \ \mathrm{on} \ x).
\end{equation}
\end{theorem}

We keep the notation of Setting \ref{setting:s} below. For each non-negative integer $j$, let $a_{1}, \cdots, a_{j}$ be arbitrary elements of $O_{X,w}$. We denote by $C_{j}$ the Koszul complex of the sequence $\{f_{1}f_{p+1}+ \varepsilon a_{1}+ \cdots + \varepsilon^{j} a_{j}, f_{2}, \cdots, f_{p}\}$. Since the support of the Koszul complex $C_{j}$ has codimension $p$, we consider $C^{j}$ as an element of $\mathcal{L}_{(-p)}(X_{j})$ which defines an element of $K_{0}((\mathcal{L}_{(-p)}(X_{j})/\mathcal{L}_{(-p-1)}(X_{j}))^{\#})_{\mathbb{Q}}$, denoted $[C_{j}]$.

When $p=1$ and $j=1$, for $X$ a surface, it was proved in Theorem 2.18 of \cite{Y-3} that the Koszul complex of $f_{1}f_{2}+ \varepsilon a_{1}$ defined a Milnor K-theoretic cycle\footnote{For the geometric meaning of Theorem 2.18 of \cite{Y-3}, we refer to page 103-104 and the summary on page 119 of Green-Griffiths \cite{GGtangentspace}. See also page 316-318 of \cite{Y-3} for a summary.}.

It is interesting to generalize Theorem 2.18 of \cite{Y-3} and find more Milnor K-theoretic cycles. Let $Q_{1}=(f_{1}, f_{2}, \cdots, f_{p})$ as in Setting \ref{setting:s} and let $z$ be point given by $Q_{2}=(f_{p+1}, f_{2}, \cdots, f_{p})$ as in Definition \ref{definition: definingCurveZ}. Under the isomorphism (3.1) (let $i=-p$)
\[
 K_{0}((\mathcal{L}_{(-p)}(X_{j})/\mathcal{L}_{(-p-1)}(X_{j}))^{\#})  \simeq \bigoplus\limits_{y \in X^{(p)}}K_{0}(O_{X_{j},y} \ \mathrm{on} \ y), 
\]
the element $[C_{j}]$ decomposes into the direct sum of two complexes $C^{1}_{j}$ and $C^{2}_{j}$
\[
[C_{j}]=C^{1}_{j}+C^{2}_{j},
\]
where $C^{1}_{j}$ and $C^{2}_{j}$ are Koszul resolutions of $(O_{X_{j},w})_{Q_{1}}/(f_{1}+ \varepsilon \dfrac{a_{1}}{f_{p+1}}+ \cdots + \varepsilon^{j}\dfrac{a_{j}}{f_{p+1}}, f_{2}, \cdots, f_{p})$ and $(O_{X_{j},w})_{Q_{2}}/(f_{p+1}+ \varepsilon \dfrac{a_{1}}{f_{1}}+ \cdots + \varepsilon^{j}\dfrac{a_{j}}{f_{1}}, f_{2}, \cdots, f_{p})$ respectively\footnote{The complex $C^{1}_{j}$ should be the Koszul resolution of $(O_{X_{j},w})_{Q_{1}}/(f_{1}+ \varepsilon \dfrac{a_{1}}{f_{p+1}}+ \cdots + \varepsilon^{j}\dfrac{a_{j}}{f_{p+1}}, \dfrac{f_{2}}{f_{p+1}}, \cdots, \dfrac{f_{p}}{f_{p+1}})$. Since $f_{p+1}^{-1}$ exists in $(O_{X_{j},w})_{(f_{1}, \cdots, f_{p})}$, we identify $C^{1}_{j}$ with Koszul resolution of $(O_{X_{j},w})_{Q_{1}}/(f_{1}+ \varepsilon \dfrac{a_{1}}{f_{p+1}}+ \cdots + \varepsilon^{j}\dfrac{a_{j}}{f_{p+1}}, f_{2}, \cdots, f_{p})$. There is a similar explanation for $C^{2}_{j}$.}.

Using isomorphisms $O_{X_{j},y}=(O_{X_{j},w})_{Q_{1}}$ and $O_{X_{j},z}=(O_{X_{j},w})_{Q_{2}}$, one sees that $C^{1}_{j} \in K_{0}(O_{X_{j},y} \ \mathrm{on} \ y)$ and $C^{2}_{j}\in K_{0}(O_{X_{j},z} \ \mathrm{on} \ z)$. Moreover, by Lemma \ref{Lemma: GilletSoule}, one sees that $C^{1}_{j} \in K^{M}_{0}(O_{X_{j},y} \ \mathrm{on} \ y)$ and $C^{2}_{j}\in K^{M}_{0}(O_{X_{j},z} \ \mathrm{on} \ z)$. In particular, when $j=0$, $C^{1}_{0}$ and $C^{2}_{0}$ are Koszul complexes of sequences $\{f_{1}, f_{2},\cdots, f_{p}\}$ and $\{f_{p+1}, f_{2},\cdots, f_{p}\}$ respectively. In other words, $C^{1}_{0}=\mu_{Y}(Y)$ and $C^{2}_{0}=\mu_{Z}(Z)$, where $\mu_{Y}(Y)$ and $\mu_{Z}(Z)$ are defined in Definition \ref{definition: map1} and Remark \ref{remark:muz}.

The following theorem gives a generalization of Theorem 2.18 of \cite{Y-3}.
 \begin{theorem} \label{theorem: TheoremKernel1}
With notation as above, $[C_{j}]=C^{1}_{j}+C^{2}_{j}$ is a Milnor K-theoretic cycle
\[
[C_{j}] =C^{1}_{j}+C^{2}_{j} \in Z^{M}_{p}(D^{\mathrm{Perf}}(X_{j})).
\]
\end{theorem}

We need to show that $[C_{j}]$ lies in the kernel of the differential $d_{1,X_{j}}^{p,-p}$. The strategy is to  use the commutative diagram (2.1) in Theorem \ref{theorem: firstorder}. Concretely, we describe the images of the Koszul complexes $C^{1}_{j}$ and $C^{2}_{j}$ under the map 
\begin{equation}
\mathrm{Ch}: \bigoplus\limits_{y \in X^{(p)}}K^{M}_{0}(O_{X_{j},y} \ \mathrm{on} \ y) \to \bigoplus\limits_{y \in X^{(p)}}H_{y}^{p}((\Omega_{X/\mathbb{Q}}^{p-1})^{\oplus j}),
\end{equation}
and then show that $\mathrm{Ch}(C^{1}_{j})+\mathrm{Ch}(C^{2}_{j})$ lies in the kernel of the differential $\partial_{1}^{p,-p}$.

\begin{proof}
When $j=1$, the map $\mathrm{Ch}$ (3.2) has been described by using a construction of Ang\'eniol and Lejeune-Jalabert \cite{A-LJ}, see Lemma 3.8 of \cite{Y-4}. For general $j$, the map Ch (3.2) can still be described by their construction in a similar way as in the case $j=1$. For our purpose, we sketch the descriptions of $\mathrm{Ch}(C^{1}_{j})$ and $\mathrm{Ch}(C^{2}_{j})$ in the following and refer to page 714-716 of \cite{Y-4} for a summary of Ang\'eniol and Lejeune-Jalabert's construction. 

Let $Q_{1}$ be the ideal $(f_{1},f_{2}, \cdots, f_{p})$ as in Setting \ref{setting:s}. The Koszul resolution of $(O_{X,w})_{Q_{1}}/(f_{1}, f_{2}, \cdots,  f_{p})$, denoted $F_{\bullet}(f_{1},f_{2}, \cdots, f_{p}) $, has the form
\[
 \begin{CD}
  0 @>>> F_{p} @>>> F_{p-1} @>>>  \cdots @>>> F_{0} @>>> 0,
 \end{CD}
\]
where each $F_{i}=\bigwedge^{i}((O_{X,w})_{Q_{1}})^{\oplus p}$. The image $\mathrm{Ch}(C^{1}_{j})$ is  represented by the following diagram,
 {\tiny
 \[
\begin{cases}
 \begin{CD}
   F_{\bullet}(f_{1}, f_{2},\cdots, f_{p}) @>>> (O_{X,w})_{Q_{1}}/(f_{1}, f_{2}, \cdots,  f_{p}) \\
  F_{p}(\cong (O_{X,w})_{Q_{1}}) @>  (\dfrac{a_{1}}{f_{p+1}}+ \cdots + \dfrac{a_{j}}{f_{p+1}})df_{2} \wedge \cdots \wedge df_{p}>> F_{0} \otimes (\Omega_{(O_{X,w})_{Q_{1}}/ \mathbb{Q}}^{p-1})^{\oplus j}(\cong(\Omega_{(O_{X,w})_{Q_{1}}/ \mathbb{Q}}^{p-1})^{\oplus j}),
 \end{CD}
 \end{cases}
\]
}where $d=d_{\mathbb{Q}}$. To be precise, the above diagram gives an element $\alpha$ in $Ext^{p}((O_{X,w})_{Q_{1}}/(f_{1}, f_{2},\cdots, f_{p}), (\Omega^{p-1}_{(O_{X,w})_{Q_{1}}/\mathbb{Q}})^{\oplus j})$. Since
\[
H_{y}^{p}((\Omega^{p-1}_{ X/\mathbb{Q}})^{\oplus j})=\varinjlim_{n \to \infty}Ext^{p}((O_{X,w})_{Q_{1}}/(f_{1}, f_{2},\cdots, f_{p})^{n}, (\Omega^{p-1}_{(O_{X,w})_{Q_{1}}/\mathbb{Q}})^{\oplus j}),
\]
the image $[\alpha]$ of $\alpha$ under the limit is in $H_{y}^{p}((\Omega^{p-1}_{X /\mathbb{Q}})^{\oplus j})$ and it is $\mathrm{Ch}(C^{1}_{j})$.

Recall that $Q_{2}$ is the ideal $(f_{p+1},f_{2}, \cdots, f_{p})$ of $O_{X,w}$. The Koszul resolution of $(O_{X,w})_{Q_{2}}/(f_{p+1}, f_{2}, \cdots,  f_{p})$, denoted $F_{\bullet}(f_{p+1},f_{2}, \cdots, f_{p}) $, has the form
\[
 \begin{CD}
  0 @>>> F'_{p} @>>> F'_{p-1} @>>>  \cdots @>>> F'_{0} @>>> 0,
 \end{CD}
\]
where each $F'_{i}=\bigwedge^{i}((O_{X,w})_{Q_{2}})^{\oplus p}$. The image $\mathrm{Ch}(C^{2}_{j})$ can be described similarly as $\mathrm{Ch}(C^{1}_{j})$ and it is  represented by the following diagram
 {\tiny
\[
\begin{cases}
 \begin{CD}
   F_{\bullet}(f_{p+1}, f_{2},\cdots, f_{p}) @>>> (O_{X,w})_{Q_{2}}/(f_{p+1}, f_{2}, \cdots,  f_{p}) \\
  F'_{p}(\cong (O_{X,w})_{Q_{2}}) @>  (\dfrac{a_{1}}{f_{1}}+ \cdots + \dfrac{a_{j}}{f_{1}})df_{2} \wedge \cdots \wedge df_{p}>> F'_{0} \otimes (\Omega_{(O_{X,w})_{Q_{2}}/ \mathbb{Q}}^{p-1})^{\oplus j}(\cong(\Omega_{(O_{X,w})_{Q_{2}}/ \mathbb{Q}}^{p-1})^{\oplus j}).
 \end{CD}
 \end{cases}
\]
}

Let $F_{\bullet}(f_{1}, f_{2},\cdots, f_{p}, f_{p+1}) $ and $F_{\bullet}(f_{p+1}, f_{2},\cdots, f_{p}, f_{1}) $ be Koszul resolutions of $O_{X, w}/(f_{1}, f_{2}, \cdots,  f_{p}, f_{p+1})$ and $O_{X, w}/(f_{p+1}, f_{2}, \cdots,  f_{p},f_{1})$ respectively. The image $\partial_{1}^{p,-p}(\mathrm{Ch}(C^{1}_{j}))$ is represented by the following diagram (denoted $\beta_{1}$)
{\tiny
\[
\begin{cases}
 \begin{CD}
   F_{\bullet}(f_{1}, f_{2},\cdots, f_{p}, f_{p+1}) @>>> O_{X,w}/(f_{1}, f_{2}, \cdots,  f_{p}, f_{p+1}) \\
  F_{p+1}(\cong O_{X,w}) @> (a_{1}+ \cdots+a_{j}) df_{2} \wedge \cdots \wedge df_{p}>> F_{0} \otimes (\Omega_{O_{X,w}/ \mathbb{Q}}^{p-1})^{\oplus j}(\cong (\Omega_{O_{X,w}/ \mathbb{Q}}^{p-1})^{\oplus j}),
 \end{CD}
 \end{cases}
\]}
and $\partial_{1}^{p,-p}(\mathrm{Ch}(C^{2}_{j}))$  is represented by the following diagram (denoted $\beta_{2}$)
{\tiny
\[
\begin{cases}
 \begin{CD}
   F_{\bullet}(f_{p+1}, f_{2},\cdots, f_{p}, f_{1}) @>>> O_{X, w}/(f_{p+1}, f_{2}, \cdots,  f_{p}, f_{1}) \\
  F_{p+1}(\cong O_{X,w}) @> (a_{1}+ \cdots+a_{j})df_{2} \wedge \cdots \wedge df_{p}>> F_{0} \otimes (\Omega_{O_{X,w}/ \mathbb{Q}}^{p-1})^{\oplus j}(\cong (\Omega_{O_{X,w}/ \mathbb{Q}}^{p-1})^{\oplus j}).
 \end{CD}
 \end{cases}
\]
}

The two complexes $F_{\bullet}(f_{1}, f_{2},\cdots, f_{p}, f_{p+1}) $ and $F_{\bullet}(f_{p+1}, f_{2},\cdots, f_{p}, f_{1}) $ are related by the following commutative diagram (see page 691 of \cite{GH})
\[
\begin{CD}
O_{X,w} @>D_{p+1}>> \wedge^{p}O_{X,w}^{\oplus p+1} @>D_{p}>> \cdots @>>> O_{X,w}^{\oplus p+1} @>D_{1}>> O_{X,w} \\
   @V\mathrm{det}A_{1}VV @V\wedge^{p}A_{1}VV @VVV  @VA_{1}VV @V=VV \\
  O_{X,w} @>E_{p+1}>> \wedge^{p}O_{X,w}^{\oplus p+1} @>E_{p}>> \cdots @>>> O_{X,w}^{\oplus p+1} @>E_{1}>> O_{X,w},
 \end{CD}
\]
where each $D_{i}$ and $E_{i}$ are defined as usual. In particular, $D_{1}=(f_{1}, f_{2},\cdots, f_{p}, f_{p+1})$, 
$E_{1}=(f_{p+1}, f_{2},\cdots, f_{p}, f_{1})$,  and $A_{1}$ is the matrix:
\[
  \left(\begin{array}{cccc}
0 & 0& 0& \cdots 1\\
0 & 1 & 0 & \cdots 0 \\
0 & 0 & 1 & \cdots 0 \\
\hdotsfor{4} \\
1 & 0 & 0 & \cdots 0
  \end{array}\right).
\]

Since the determinant $\mathrm{det}A_{1}=-1$, one has
{\small
\[
\beta_{1}=-\beta_{2} \in Ext^{p+1}(O_{X, w}/(f_{1}, f_{2}, \cdots,  f_{p}, f_{p+1}), (\Omega_{O_{X,w}/ \mathbb{Q}}^{p-1})^{\oplus j}).
\]
}Consequently, $\partial_{1}^{p,-p}(\mathrm{Ch}(C^{1}_{j}))+\partial_{1}^{p,-p}(\mathrm{Ch}(C^{2}_{j}))=0 \in H^{p+1}_{w}(\Omega_{O_{X,w}/ \mathbb{Q}}^{p-1})$. This implies that $d^{p,-p}_{1,X_{j}}(C^{1}_{j}+C^{2}_{j}) = 0$ because of the commutative diagram (2.1)
\[
  \begin{CD}
     \mathrm{Ch}(C^{1}_{j})+\mathrm{Ch}(C^{2}_{j})  @<\mathrm{Ch}<<  C^{1}_{j}+C^{2}_{j} \\
      @V\partial_{1}^{p,-p}VV @Vd_{1,X_{j}}^{p,-p}VV  \\
      \partial_{1}^{p,-p}(\mathrm{Ch}(C^{1}_{j})+\mathrm{Ch}(C^{2}_{j}))=0 @<\mathrm{Ch}< \cong < d_{1,X_{j}}^{p,-p}(C^{1}_{j}+C^{2}_{j}). \\
  \end{CD}
\]

In conclusion, $C^{1}_{j}+C^{2}_{j}$ is a Milnor K-theoretic cycle in the sense of Definition \ref{definition: Milnor K-theoretic Chow groups}.
\end{proof}

For each integer $j$, it is obvious that $g^{*}_{j}([C_{j}])=[C_{j-1}]$, where $g^{*}_{j}$ is the map (2.2). When $j=1$, $g^{*}_{1}([C_{1}])=g^{*}_{1}(C^{1}_{1}+C^{2}_{1})=\mu_{Y}(Y)+\mu_{Z}(Z)$. This shows that
\begin{corollary} \label{corollary:Ckernel}
With notation as above, $[C_{1}] \in Z^{M}_{p}(D^{\mathrm{Perf}}(X_{1}))$ is a first order deformation of $\mu_{Y}(Y)+\mu_{Z}(Z)$ and it can be successively lifted to higher order $[C_{j}] \in Z^{M}_{p}(D^{\mathrm{Perf}}(X_{j}))$.
\end{corollary}

For a first order infinitesimal deformation $Y^{1}$ of $Y$, by Definition \ref{definition: map1}, $\mu_{Y}(Y^{1}) \in K_{0}(O_{X_{1},y} \ \mathrm{on} \ y)_{\mathbb{Q}}$ is given by the Koszul complex $F_{\bullet}(f_{1}+\varepsilon g_{1}, \cdots, f_{p}+\varepsilon g_{p})$. We want to check whether $\mu_{Y}(Y^{1})$ is a Milnor K-theoretic cycles or not.

For simplicity, we assume that $g_{2}= \cdots = g_{p}=0$ in the following. In notation of Setting \ref{setting:s}, for $g_{1} \in O_{X,y}= (O_{X,w})_{Q_{1}}$, we write $g_{1}=\dfrac{a}{b} \in O_{X,y}$, where $a,b \in O_{X,w}$ and $b \notin Q_{1}$, then $b$ is either in or not in the maximal idea $(f_{1}, f_{2}, \cdots, f_{p}) \subset O_{X,w}$.

\begin{lemma}  \label{lemma: trivialdeform}
With notation as above, if $b_{1} \notin (f_{1}, f_{2}, \cdots, f_{p})$, then
\begin{itemize}
\item [1.] $\mu_{Y}(Y^{1})$ is a Milnor K-theoretic cycle which lifts $\mu_{Y}(Y)$;
\item[2.] $\mu_{Y}(Y^{1})$ lifts to higher order in $Z^{M}_{p}(D^{\mathrm{Perf}}(X_{j}))$ successively in the sense of Definition \ref{definition: deformation}.
\end{itemize}

\end{lemma}

\begin{proof}

If $b \notin (f_{1}, f_{2}, \cdots, f_{p}, f_{p+1})$, then $b$ is a unit in $O_{X,w}$, this says $g_{1}= \dfrac{a}{b} \in O_{X,w}$. For each integer $j \geq 1$, let $T^{j}$ denote Koszul resolution of $(O_{X_{j},w})_{Q_{1}}/(f_{1}+ \varepsilon g_{1} + \varepsilon^{2} h_{2} + \cdots + \varepsilon^{j} h_{j}, f_{2}, \cdots, f_{p})$, where $h_{2}, \cdots, h_{j}$ are arbitrary elements of $O_{X,w}$. It is obvious that $T^{j} \in K_{0}(O_{X_{j},y} \ \mathrm{on} \ y)_{\mathbb{Q}}$.

By Lemma \ref{Lemma: GilletSoule}, $T^{j} \in K^{M}_{0}(O_{X_{j},y} \ \mathrm{on} \ y)$. We use the same strategy of proving Theorem \ref{theorem: TheoremKernel1} to prove that $T^{j} \in Z^{M}_{p}(D^{\mathrm{Perf}}(X_{j}))$. As explained in the proof of Theorem \ref{theorem: TheoremKernel1}, the image of $T^{j}$ under the Ch map (3.2), $\mathrm{Ch}(T^{j})$, can be represented by the following diagram
{\tiny
\[
\begin{cases}
 \begin{CD}
   F_{\bullet}(f_{1}, f_{2},\cdots, f_{p}) @>>> (O_{X,w})_{Q_{1}}/(f_{1}, f_{2}, \cdots,  f_{p}) \\
  F_{p}(\cong (O_{X,w})_{Q_{1}}) @>  (g_{1}+h_{2}+\cdots + h_{j})df_{2} \wedge \cdots \wedge df_{p}>> F_{0} \otimes (\Omega_{(O_{X,w})_{Q_{1}}/ \mathbb{Q}}^{p-1})^{\oplus j}(\cong (\Omega_{(O_{X,w})_{Q_{1}}/ \mathbb{Q}}^{p-1})^{\oplus j}).
 \end{CD}
 \end{cases}
\]
}Here $F_{\bullet}(f_{1}, f_{2},\cdots, f_{p})$ is the Koszul resolution of $(O_{X,w})_{Q_{1}}/(f_{1}, f_{2}, \cdots,  f_{p})$. Since $f_{p+1} \notin Q_{1}$, $f_{p+1}^{-1}$ exists in $(O_{X,w})_{Q_{1}}$, $(g_{1}+h_{2+}\cdots +h_{j})df_{2} \wedge \cdots \wedge df_{p}$ can be rewritten as
\[
(g_{1}+h_{2+}\cdots +h_{j})df_{2} \wedge \cdots \wedge df_{p}= \dfrac{(g_{1}+h_{2}+\cdots +h_{j}) f_{p+1}}{f_{p+1}}df_{2} \wedge \cdots \wedge df_{p}.
\]

The image $\partial_{1}^{p,-p}(\mathrm{Ch}(T^{j}))$ is represented by the following diagram (denoted $\gamma$),
{\tiny
\[
\begin{cases}
 \begin{CD}
   F_{\bullet}(f_{1}, f_{2},\cdots, f_{p}, f_{p+1}) @>>> O_{X, w}/(f_{1}, f_{2}, \cdots,  f_{p}, f_{p+1}) \\
  F_{p+1}(\cong O_{X,w}) @>  (g_{1}+h_{2}+\cdots +h_{j})f_{p+1}df_{2} \wedge \cdots \wedge df_{p}>> F_{0} \otimes (\Omega_{O_{X,w}/ \mathbb{Q}}^{p-1})^{\oplus j}(\cong(\Omega_{O_{X,w}/ \mathbb{Q}}^{p-1})^{\oplus j}),
 \end{CD}
 \end{cases}
\]
}where the complex $F_{\bullet}(f_{1}, f_{2},\cdots, f_{p}, f_{p+1})$ is of the form
\[
 \begin{CD}
  0 @>>>  \bigwedge^{p+1}(O_{X, w})^{\oplus p+1} @>M_{p+1}>> \bigwedge^{p}(O_{X, w})^{\oplus p+1} @>>>  \cdots.
 \end{CD}
\]
Let $\{e_{1}, \cdots, e_{p+1} \}$ be a basis of $(O_{X, w})^{\oplus p+1} $, the map $M_{p+1}$ is 
\[
 e_{1}\wedge \cdots \wedge e_{p+1}  \to \sum^{p+1}_{j=1}(-1)^{j}f_{j}e_{1}\wedge \cdots \wedge \hat{e_{j}} \wedge \cdots e_{p+1},
\]
where $\hat{e_{j}}$ means to omit $e_{j}$.

Since $f_{p+1}$ appears in $M_{p+1}$, the diagram $\gamma$ defines a trivial element of $ Ext^{p+1}(O_{X, w}/(f_{1}, \cdots,  f_{p}, f_{p+1}), (\Omega_{O_{X,w}/ \mathbb{Q}}^{p-1})^{\oplus j})$. Hence, $\partial_{1}^{p,-p}(\mathrm{Ch}(T^{j}))=0$. It follows from the commutative diagram (2.1) that $d^{p,-p}_{1,X_{j}}(T^{j}) = 0$. 
This proves that $T^{j} \in Z^{M}_{p}(D^{\mathrm{Perf}}(X_{j}))$. 

It is obvious that $g^{*}_{j}(T^{j})=T^{j-1}$, where $g^{*}_{j}$ is the map (2.2). In particular, $T^{1}= \mu_{Y}(Y^{1})$ and $g^{*}_{1}(T^{1})=g^{*}_{1}(\mu_{Y}(Y^{1}))=\mu_{Y}(Y)$.

\end{proof}

Now, we consider the case $b \in (f_{1}, f_{2}, \cdots, f_{p}, f_{p+1})$. Since $b \notin (f_{1}, f_{2}, \cdots, f_{p})$,
we can write $b= \sum_{i=1}^{p} a_{i}f_{i}^{n_{i}} + a_{p+1}f_{p+1}^{n_{p+1}}$, where $a_{p+1}$ is a unit in 
$O_{X,w}$ and each $n_{j}$ is some integer. For simplicity, we assume that each $n_{j} =1$ and $a_{p+1}=1$.

Let $K^{M}_{0}(O_{X_{1},y} \ \mathrm{on} \ y, \varepsilon)$ denote the kernel of the natural projection
\[
  K^{M}_{0}(O_{X_{1},y} \ \mathrm{on} \ y)  \xrightarrow{\varepsilon=0} 
  K^{M}_{0}(O_{X,y} \ \mathrm{on} \ y).
  \] 
There exists the following isomorphism (see Corollary 9.5 in \cite{DHY} or Corollary 3.11 in \cite{Y-2})
\[
K^{M}_{0}(O_{X_{1},y} \ \mathrm{on} \ y, \varepsilon) \cong H_{y}^{p}(\Omega^{p-1}_{X/\mathbb{Q}}).
 \] It follows that there is an isomorphism
 \begin{equation}
 (\mathrm{P}, \mathrm{Ch}): K^{M}_{0}(O_{X_{1},y} \ \mathrm{on} \ y) \cong K^{M}_{0}(O_{X,y} \ \mathrm{on} \ y) \oplus H_{y}^{p}(\Omega^{p-1}_{X/\mathbb{Q}}),
 \end{equation}
where $\mathrm{P}$ is induced by the map $\varepsilon \to 0$ and $\mathrm{Ch}$ is the map induced by Chern character from K-theory to negative cyclic homology (see Theorem \ref{theorem: firstorder}). 

For $\mu_{Y}(Y^{1})=F_{\bullet}(f_{1}+\varepsilon g_{1}, f_{2},\cdots, f_{p}) \in K^{M}_{0}(O_{X_{1},y} \ \mathrm{on} \ y)$, where $g_{1}=\dfrac{a}{b}$, the image $\mathrm{P}(\mu_{Y}(Y^{1}))=\mu_{Y}(Y) \in K^{M}_{0}(O_{X,y} \ \mathrm{on} \ y)$. The description of the image $\mathrm{Ch}(\mu_{Y}(Y^{1}))$ has been recalled in the proof of Theorem \ref{theorem: TheoremKernel1}. Concretely, let $F_{\bullet}(f_{1}, f_{2},\cdots, f_{p})$ be the Koszul resolution of $(O_{X,w})_{Q_{1}}/(f_{1}, f_{2}, \cdots,  f_{p})$, which is of the form
\[
 \begin{CD}
  0 @>>> F_{p} @>M_{p}>> F_{p-1} @>>>  \cdots @>>> F_{0} @>>> 0,
 \end{CD}
\]
where each $F_{i}=\bigwedge^{i}((O_{X,w})_{Q_{1}})^{\oplus p}$. The map $M_{p}$ is 
\begin{equation}
 e_{1}\wedge \cdots \wedge e_{p}  \to \sum^{p}_{j=1}(-1)^{j}f_{j}e_{1}\wedge \cdots \wedge \hat{e_{j}} \wedge \cdots e_{p},
\end{equation}
where $\{e_{1}, \cdots, e_{p} \}$ is a basis of $((O_{X, w})_{Q_{1}})^{\oplus p} $ and $\hat{e_{j}}$ means to omit $e_{j}$.

The following diagram (denoted $\gamma_{1}$)
\[
\begin{cases}
 \begin{CD}
   F_{\bullet}(f_{1}, f_{2},\cdots, f_{p}) @>>> (O_{X,w})_{Q_{1}}/(f_{1}, f_{2}, \cdots,  f_{p}) \\
  F_{p}(\cong (O_{X,w})_{Q_{1}}) @> \dfrac{a}{b} df_{2} \wedge \cdots \wedge df_{p}>> F_{0} \otimes \Omega_{(O_{X,w})_{Q_{1}}/ \mathbb{Q}}^{p-1}(\cong\Omega_{(O_{X,w})_{Q_{1}}/ \mathbb{Q}}^{p-1}),
 \end{CD}
 \end{cases}
\]
 defines an element of $Ext^{p}((O_{X,w})_{Q_{1}}/(f_{1}, f_{2},\cdots, f_{p}), \Omega^{p-1}_{(O_{X,w})_{Q_{1}}/\mathbb{Q}})$. The limit $[\gamma_{1}] \in H_{y}^{p}(\Omega^{p-1}_{(O_{X,w})_{Q_{1}}/\mathbb{Q}})$ of $\gamma_{1}$ is $\mathrm{Ch}(\mu_{Y}(Y^{1}))$.

Let $F(f_{1}+\varepsilon \dfrac{a_{1}}{f_{p+1}}, f_{2}, \cdots, f_{p})$ be the Koszul resolution of $(O_{X_{1},w})_{Q_{1}}/(f_{1}+\varepsilon \dfrac{a_{1}}{f_{p+1}}, f_{2}, \cdots,  f_{p})$, by Lemma \ref{Lemma: GilletSoule}, it can be considered as an element of $K^{M}_{0}(O_{X_{1},y} \ \mathrm{on} \ y)$. Its images under the map $\mathrm{P}$ is the complex $\mu_{Y}(Y)$ and its image under the map $\mathrm{Ch}$ is the limit $[\gamma_{2}] \in H_{y}^{p}(\Omega^{p-1}_{(O_{X,w})_{Q_{1}}/\mathbb{Q}})$, where $\gamma_{2}$ is the following diagram
 \[
\begin{cases}
 \begin{CD}
   F_{\bullet}(f_{1}, f_{2},\cdots, f_{p}) @>>> (O_{X,w})_{Q_{1}}/(f_{1}, f_{2}, \cdots,  f_{p}) \\
  F_{p}(\cong (O_{X,w})_{Q_{1}}) @> \dfrac{a}{f_{p+1}} df_{2} \wedge \cdots \wedge df_{p}>> F_{0} \otimes \Omega_{(O_{X,w})_{Q_{1}}/ \mathbb{Q}}^{p-1}(\cong\Omega_{(O_{X,w})_{Q_{1}}/ \mathbb{Q}}^{p-1}).
 \end{CD}
 \end{cases}
\]

It follows from the isomorphism (3.3) that
{\small 
\[
\mu_{Y}(Y^{1}) \cong \mu_{Y}(Y)+ [\gamma_{1}], \ F(f_{1}+\varepsilon \dfrac{a_{1}}{f_{p+1}}, f_{2}, \cdots, f_{p}) \cong \mu_{Y}(Y)+ [\gamma_{2}].
\]
}Since $\dfrac{a}{b}-\dfrac{a}{f_{p+1}}=\dfrac{a(-\sum_{i=1}^{p}a_{i}f_{i})}{bf_{p+1}}$ and each $f_{i}$ ($i=1, \cdots, p$) appears in the map $M_{p}$ (3.4), $[\gamma_{1}]= [\gamma_{2}] \in H_{y}^{p}(\Omega^{p-1}_{(O_{X,w})_{Q_{1}}/\mathbb{Q}})$.  This shows that
\begin{lemma} \label{lemma: complexAgree}
The complexes $\mu_{Y}(Y^{1})$ and $F(f_{1}+\varepsilon \dfrac{a_{1}}{f_{p+1}}, f_{2}, \cdots, f_{p})$ define the same element of $K^{M}_{0}(O_{X_{1},y} \ \mathrm{on} \ y)$.
\end{lemma}

It sufficient to assume that $\mu_{Y}(Y^{1})$ is the Koszul complex $F(f_{1}+\varepsilon \dfrac{a_{1}}{f_{p+1}}, f_{2}, \cdots, f_{p})$ in the following. The image $\partial_{1}^{p,-p}(\mathrm{Ch}(\mu_{Y}(Y^{1})))$ is represented by the following diagram
\[
\begin{cases}
 \begin{CD}
   F_{\bullet}(f_{1}, f_{2},\cdots, f_{p}, f_{p+1}) @>>> O_{X,w}/(f_{1}, f_{2}, \cdots,  f_{p}, f_{p+1}) \\
  F_{p+1}(\cong O_{X,w}) @> a df_{2} \wedge \cdots \wedge df_{p}>> F_{0} \otimes \Omega_{O_{X,w}/ \mathbb{Q}}^{p-1}(\cong\Omega_{O_{X,w}/ \mathbb{Q}}^{p-1}),
 \end{CD}
 \end{cases}
\]
which is not trivial. It follows from the commutative diagram (2.1) that $\mu_{Y}(Y^{1})$ is not a Milnor K-theoretic cycles. Hence, $\mu_{Y}(Y^{1})$ is not a deformation of $\mu_{Y}(Y)$. In this way, obstructions to deforming $\mu_{Y}(Y)$ arise.

Now, it is time to use Idea \ref{idea: i} to eliminate the obstructions. Recall that $Z$ is the subscheme defined in Definition ~\ref{definition: definingCurveZ} and $\mu_{Z}(Z) \in Z^{M}_{p}(D^{\mathrm{Perf}}(X))$ is the Koszul complex of the sequence $\{f_{p+1}, f_{2},\cdots, f_{p}\}$. We use $\mu_{Z}(Z)$ to eliminate obstructions to deforming $\mu_{Y}(Y)$.

Since $\mu_{Y}(Y)$ can be written as a formal sum
\[
\mu_{Y}(Y)=(\mu_{Y}(Y)+\mu_{Z}(Z))-\mu_{Z}(Z) \in Z^{M}_{p}(D^{\mathrm{Perf}}(X)),
\] 
to lift $\mu_{Y}(Y)$ is equivalent to lifting $\mu_{Y}(Y)+\mu_{Z}(Z)$ and $\mu_{Z}(Z)$ respectively.

By Corollary \ref{corollary:Ckernel}, the element $[C_{1}]$ is a Milnor K-theoretic cycle and it is a first order deformation of $\mu_{Y}(Y)+\mu_{Z}(Z)$. By Remark \ref{Remark:}, $\mu_{Z}(Z) \in Z^{M}_{p}(D^{\mathrm{Perf}}(X_{1}))$, so the formal sum $[C_{1}]-\mu_{Z}(Z) \in Z^{M}_{p}(D^{\mathrm{Perf}}(X_{1}))$. Since $g^{*}_{1}([C_{1}]-\mu_{Z}(Z))=(\mu_{Y}(Y)+\mu_{Z}(Z))-\mu_{Z}(Z)=\mu_{Y}(Y)$, where $g^{*}_{1}$ is the map (2.2), $[C_{1}]-\mu_{Z}(Z)$ is a first order deformation of $\mu_{Y}(Y)$.

The Milnor K-theoretic cycle $[C_{1}]-\mu_{Z}(Z)$ lies in the direct sum of K-groups
\[
[C_{1}]-\mu_{Z}(Z) \in Z^{M}_{p}(D^{\mathrm{Perf}}(X_{1})) \subset \bigoplus_{y \in X^{(p)}} K^{M}_{0}(O_{X_{1},y} \ \mathrm{on} \ y).
\]
Let $([C_{1}]-\mu_{Z}(Z))|_{Y}$ denote the direct summand corresponding to $Y$ (with generic point $y$) of $[C_{1}]-\mu_{Z}(Z)$, one sees that $([C_{1}]-\mu_{Z}(Z))|_{Y}=\mu_{Y}(Y^1)$.

By Remark \ref{Remark:}, for each integer $j >1$, $\mu_{Z}(Z) \in Z^{M}_{p}(D^{\mathrm{Perf}}(X_{j}))$. According to Corollary \ref{corollary:Ckernel}, the element $[C_{1}] \in Z^{M}_{p}(D^{\mathrm{Perf}}(X_{1}))$ lifts to $[C_{j}] \in Z^{M}_{p}(D^{\mathrm{Perf}}(X_{j}))$ successively. It follows that $[C_{1}]-\mu_{Z}(Z) \in Z^{M}_{p}(D^{\mathrm{Perf}}(X_{1}))$ lifts to $[C_{j}]-\mu_{Z}(Z) \in Z^{M}_{p}(D^{\mathrm{Perf}}(X_{j}))$ successively. In summary,
\begin{theorem} \label{theorem: mainTheorem}
With notation as above, in the case $b \in (f_{1}, \cdots,f_{p+1})$, there exists a Milnor K-theoretic cycle $\mu_{Z}(Z) \in Z^{M}_{p}(D^{\mathrm{Perf}}(X))$, where $Z \subset X$ is another irreducible closed subscheme of codimension $p$, and a Milnor K-theoretic cycle $[C_{1}] \in Z^{M}_{p}(D^{\mathrm{Perf}}(X_{1}))$, which is a first order deformation of $\mu_{Y}(Y) + \mu_{Z}(Z)$ such that
\begin{itemize}
\item [1.] $([C_{1}]-\mu_{Z}(Z))|_{Y}=\mu_{Y}(Y^1)$;
\item [2.] $[C_{1}]-\mu_{Z}(Z)$ is a first order deformation of $\mu_{Y}(Y)$;
\item [3.] $[C_{1}]-\mu_{Z}(Z)$ lifts to higher order successively.
\end{itemize}

\end{theorem}

\textbf{Acknowledgements.} This paper is an extension of \cite{Y-5}. The author thanks Spencer Bloch, Jerome William Hoffman and Christophe Soul\'e for discussions. He also thanks Shiu-Yuen Cheng and Congling Qiu for help.

The idea to eliminate obstructions to deforming divisors had been known to Mark Green and Phillip Griffiths.


\end{document}